\documentclass[11pt]{amsart}
\usepackage{geometry}                
\usepackage{graphicx}
\usepackage{amssymb}
\usepackage{epstopdf}
\usepackage{bbm}
\usepackage{fonts}
\usepackage{nicetikz}
\usepackage{subcaption}
\usepackage{orcidlink}

\usepackage{array, multirow, makecell}

\newcolumntype{C}{>{\displaystyle}>{$}c<{$}}
\newcolumntype{L}{>{\displaystyle}>{$}l<{$}}
\newcolumntype{R}{>{\displaystyle}>{$}r<{$}}

\DeclareMathOperator{\codim}{codim}

\DeclareMathOperator{\id}{id}

\DeclareMathOperator{\ord}{ord}

\DeclareMathOperator{\supp}{supp}
\DeclareMathOperator{\ic}{ic}
\newcommand{\p}{\partial}

\newtheorem{conjecture}{Conjecture}[section]
\newtheorem{corollary}[conjecture]{Corollary}
\newtheorem{definition}[conjecture]{Definition}
\newtheorem{example}[conjecture]{Example}
\newtheorem{lemma}[conjecture]{Lemma}
\newtheorem{proposition}[conjecture]{Proposition}
\newtheorem{remark}[conjecture]{Remark}
\newtheorem{theorem}[conjecture]{Theorem}

\title{Counting $D_4$ singularities in the image of a wave front}
\author{C. Mu\~noz-Cabello \orcidlink{0000-0002-3115-1851}, J.J. Nu\~no-Ballesteros \orcidlink{0000-0001-6725-541X}, R. Oset Sinha \orcidlink{0000-0002-5652-7982}}

\address{Departament de Matem\`{a}tiques,
Universitat de Val\`encia, Campus de Burjassot, 46100 Burjassot,
Spain}
\email{Christian.Munoz@uv.es}
\email{Juan.Nuno@uv.es}
\email{Raul.Oset@uv.es}

\thanks{Work of C. Mu\~noz-Cabello, Juan J. Nu\~no-Ballesteros and R. Oset Sinha partially supported by Grant PID2021-124577NB-I00 funded by MCIN/AEI/ 10.13039/501100011033 and by ``ERDF A way of making Europe"}


\begin{document}
\maketitle

\begin{abstract}
	We give a formula to count the number of $D_4$ singularities in a stable frontal perturbation of a corank $2$ wave front singularity $f\colon (\mathbb{C}^3,0) \to (\mathbb{C}^4,0)$ using Mond's method of stable perturbations of map germs.
	For a generic germ of corank $2$ wave front $f\colon (\mathbb{C}^3,S) \to (\mathbb{C}^4,0)$, the image of a stable deformation $f_t$ of $f$ exhibits $A_k$ singularities with $k \leq 4$, their transverse intersections and the aforementioned $D_4$ singularities for $0 < |t| \ll 1$.
	By interpreting the image of $f_t$ as the discriminant (the image of the critical point set) of a smooth map germ $H_t\colon (\mathbb{C}^5,0) \to (\mathbb{C}^4,0)$, we define an algebra whose dimension over $\mathbb{C}$ is equal to the number of $D_4$ points in the image of $f_t$.
\end{abstract}

\section{Introduction}
A wave front is a smooth map germ $f\colon N^n \to Z^{n+1}$, with $N$, $Z$ complex manifolds, which can be lifted onto a Legendrian immersion $\tilde f\colon N \looparrowright PT^*Z$ via a Legendrian fibration $\pi\colon PT^*Z \to Z$, where $PT^*Z$ is the space of contact elements on $Z$.
For $n=3$, a generic frontal perturbation of $f$ (in the sense of Definition \ref{frontal unfolding}; see also \cite{MNO_Mond}) exhibits isolated $D_4$ singularities near the origin, similar to how quadruple points manifest on a generic perturbation of a holomorphic mapping $U \subseteq \mb{C}^3 \to \mb{C}^4$.
Our goal is to provide a formula to count the number of $D_4$ singularities emerging in this arrangement.

The problem of counting isolated singularities in a stable deformation of a map germ can be traced back to \cite{Mond_Classification}, where Mond gives a formula to count the number of cross cap singularities in a stable $1$-parameter deformation $(f_t)$ of $f$, which he calls a \emph{generic perturbation} of $f$.
This method inspired a series of articles in the following years \cite{FukudaIshikawa, MondGaffney, MondGaffney2, Rieger}, in which the authors apply Mond's method of generic perturbations to smooth map germs $G\colon (\mb{C}^2,0) \to (\mb{C}^2,0)$, with the aim of finding formulas and relations between the isolated singularities in the discriminant curve of $G$.
Fukui, Nuño-Ballesteros and Saia also define the $c_\mathbf{i}(f)$ invariant in \cite{FNBS_95, FNBS_98}, which counts the number of isolated Thom-Boardman singularities of type $\Sigma^\mathbf{i}$ for a smooth $f\colon (\mb{C}^n,0) \to (\mb{C}^p,0)$.
The results regarding discriminant curves of smooth map germs were later extended by Marar, Montaldi and Ruas in \cite{MMR_Schemes}, where they present a general method to count singularities of type $A_{j_1}\dots A_{j_k}$ with $j_1+\dots+j_k=n$ in the discriminant of a corank $1$ map germ $G\colon (\mb{R}^n,0) \to (\mb{R}^n,0)$.

Arnold and his colleagues from the Russian school of singularities \cite{Arnold_I} proved that every Legendrian submanifold can be written in terms of a smooth function $g\colon \mb{C}^d \to \mb{C}$ with smooth critical set.
In our setting, this means that every wave front $f\colon (\mb{C}^n,0) \to (\mb{C}^{n+1},0)$ can be written as the discriminant of a corank $1$ map germ $G\colon (\mb{C}^{n+d},0) \to (\mb{C}^{n+1},0)$ with smooth critical set for some $d > 0$.
Using this characterisation of wave fronts and the formulas in \cite{MMR_Schemes}, we are capable of counting the number of corank $1$ singularities emerging on a generic frontal perturbation of $f$.

The number of isolated singularities of a given type in a frontal perturbation of $f$ is not only an analytic invariant, but can also be used to compute other numerical invariants, such as the frontal Milnor number of $f$ (see \cite{MNO_Mond}).
In \cite{MNO_Surfaces}, we use the number of isolated singularities of a frontal surface to give an expression for the frontal Minor number of $f$ (see also Example \ref{swallowtail}) and the Milnor numbers of the image curves of singular and transverse double points.
We also related the frontal Milnor number of an analytic plane curve to its image Milnor number in \cite{MNO_Frontals}, using the number of singular points in a generic frontal perturbation.

For $n=3$, we see the emergence of the first corank $2$ singularity in the classification of stable wave fronts, namely $D_4$, which can be defined as any map germ $f\colon (\mb{C}^3,0) \to (\mb{C}^4,0)$ which is $\ms{A}$-equivalent to 
\begin{equation}\label{D4 parametrisation}
	f(u,v,w)=(u, v w, 2 u v + 3 v^2 + w^2, u v^2 + 2 v^3 + 2 v w^2).
\end{equation}
This complex singularity splits into two real singularities, which Arnold denotes as $D_4^+$ and $D_4^-$.
On Figure \ref{fig: D4}, we have illustrated the projections of $D_4^+$ and $D_4^-$ onto the hyperplane of equation $X=0$ (in coordinates $(X,Y,Z,T)$) for different values of $u$.
These illustrations are reproductions of those found in \cite{Arnold_I}, \S 22.1.

\begin{figure}[ht]
	\includegraphics[width=.8\textwidth]{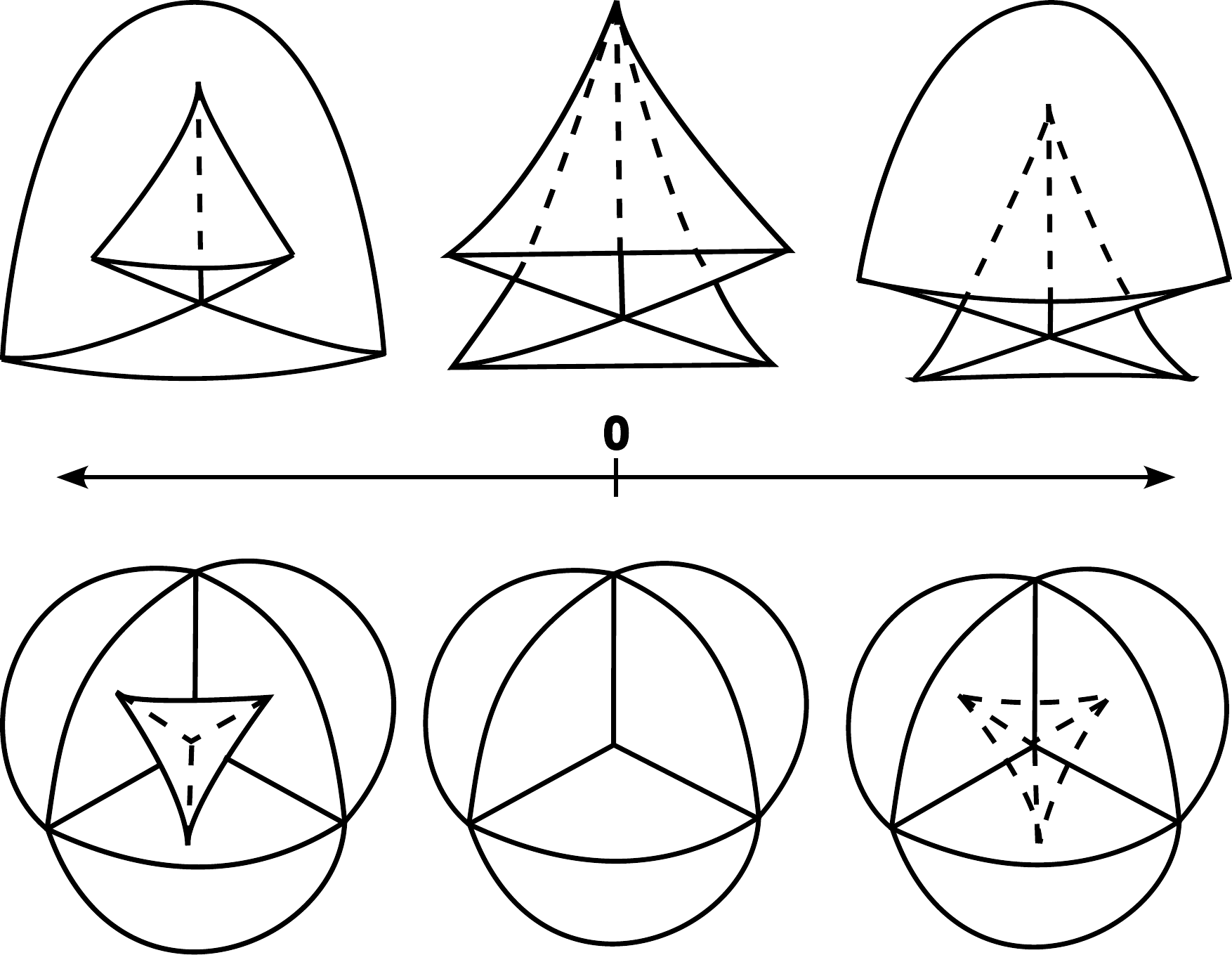}
	\caption{Projections of the $D_4^+$ (top) and $D_4^-$ (bottom) singularities onto the hyperplane $X=0$.
	Source: \cite{Arnold_I}, \S 22.1. \label{fig: D4}}
\end{figure}
	
In Section 2, we present some background information on germs of wave fronts using the theory of deformations of frontal map germs we developed in \cite{MNO_Frontals}.

In Section 3, we define for a corank $2$ wave front $f\colon (\mb{C}^3,0) \to (\mb{C}^4,0)$ given in the form $f(u,v,w)=(u,p(u,v,w),q(u,v,w),r(u,v,w))$ the algebra
	\[\frac{\ms{O}_3}{\langle p_v, p_w,q_v,q_w\rangle},\]
where $g_v$, $g_q$ denote the partial derivatives of the function $g$ with respect to $v$ and $w$, and show that the $\mb{C}$-dimension of this space is equal to the number of $D_4$ points in a generic frontal perturbation of $f$.
The authors are unaware of any other formula in the literature that counts the number of isolated corank $2$ singularities on the image of a wave front.
Combining this algebra with the formulas given in \cite{MMR_Schemes}, one can then count any isolated singularity in a generic frontal perturbation of $f$ by finding the associated map germ $G\colon (\mb{C}^{n+2},0) \to (\mb{C}^{n+1},0)$, a procedure which has been detailed in \cite{MNO_Mond}.

The idea for this article was initially proposed by professor T. Ohmoto to the first author, asking whether it is possible to give a formula to count the number of $D_4$ singularities in a generic frontal perturbation of a frontal map germ $f\colon (\mb{C}^3,0) \to (\mb{C}^4,0)$.
Since the list of corank $2$ stable singularities of frontals is not known in dimension $3$, we decided to narrow the scope of the formula to only germs of wave fronts, resulting in this article. 

\section{Wave fronts}
Let $PT^*\mb{C}^{n+1}$ be the projectivized cotangent bundle of $\mb{C}^{n+1}$.
If $(z,[\omega]) \in PT^*\mb{C}^{n+1}$, we equip $PT^*\mb{C}^{n+1}$ with the contact structure given by the differential form
	\[\alpha=\omega_1\,dz^1+\dots+\omega_{n+1}\,dz^{n+1}.\]
We consider the canonical projection $\pi\colon PT^*\mb{C}^{n+1} \to \mb{C}^{n+1}$ given by $\pi(z,[\omega])=z$, whose fibres are Legendrian submanifolds of $PT^*\mb{C}^{n+1}$ under this contact structure (i.e. $\ker d\pi_{(z,[\omega])} \subseteq \ker\alpha_{(z,[\omega])}$ for all $(z,[\omega]) \in PT^*\mb{C}^{n+1}$).
A holomorphic map $F\colon N \subset \mb{C}^n \to PT^*\mb{C}^{n+1}$ is \textbf{integral} if $F^*\alpha=0$.

\begin{definition}
	Let $N \subset \mb{C}^n$ be an open subset.
	A holomorphic map $f\colon N \to \mb{C}^{n+1}$ is \textbf{frontal} if there exists an integral map $F\colon N \to PT^*\mb{C}^{n+1}$ such that
		\[f=\pi \circ F.\]
	If $F$ is an immersion, we say $f$ is a \textbf{Legendrian map} or \textbf{wave front}.
	Similarly, a hypersurface $X \subset \mb{C}^{n+1}$ is \textbf{frontal} (resp. a \textbf{wave front}) if there exists a frontal map (resp. wave front) $f\colon N \to \mb{C}^{n+1}$ such that $X=f(N)$.
\end{definition}

If $F\colon N \to PT^*\mb{C}^{n+1}$ is an integral map and $f=\pi\circ F$,
	\[0=F^*\alpha=\sum^{n+1}_{i=1}\nu_i d(Z_i\circ F)=\sum^{n+1}_{i=1}\sum^n_{j=1}\nu_i\frac{\p f_i}{\p x_j}\,dx^j\]
for some $\nu_1,\dots,\nu_{n+1}\colon (\mb{C}^n,0) \to \mb{C}$, not all of them zero.
This is the same as claiming that there exists a nowhere vanishing smooth section $\nu\colon N \to T^*\mb{C}^{n+1}$ with $\pi\circ \nu=f$ such that $\nu(df\circ \xi)=0$ for all vector fields $\xi$ on $N$.

Since $PT^*\mb{C}^{n+1}$ is a fibre bundle, we can find for each $(z,[\omega]) \in PT^*\mb{C}^{n+1}$ an open neighbourhood $Z \subset \mb{C}^{n+1}$ of $z$ and an open $U \subseteq \mb{C}P^n$ such that $\pi^{-1}(Z)\cong Z\times U$.
Therefore, $F$ is contact equivalent to the integral mapping $\tilde f(x)=(f(x),[\nu_x])$, known as the \textbf{Nash lift} of $f$.
In particular, if $\Sigma(f)$ is nowhere dense in $N$, $\nu$ can be uniquely extended from $N\backslash \Sigma(f)$ to $N$, and there is a one-to-one correspondence between $f$ and $\tilde f$.
Such a frontal map is known as a \textbf{proper frontal} map \cite{Ishikawa_Survey}.

If $f\colon N \to \mb{C}^{n+1}$ is a proper frontal map with Nash lift $\tilde f$, we denote the corank of $\tilde f$ as $\ic(f)$, and call it the \textbf{integral corank} of $f$.

\begin{definition}
	Let $S \subset \mb{C}^n$ be a finite set.
	A holomorphic multigerm $f\colon (\mb{C}^n,S) \to (\mb{C}^{n+1},0)$ is \textbf{frontal} if it has a frontal representative $f\colon N \to Z$.
	Given a hypersurface $X \subset \mb{C}^{n+1}$, $(X,0)$ is a \textbf{frontal} hypersurface germ if there exists a finite frontal map germ $f\colon (\mb{C}^n,S) \to (\mb{C}^{n+1},0)$ such that $(X,0)=f(\mb{C}^n,S)$.
\end{definition}

Similarly, we shall say that $f$ is a proper frontal multigerm if it has a representative which is a proper frontal map.

\begin{definition}
	A frontal map germ $f\colon (\mb{C}^n,S) \to (\mb{C}^{n+1},0)$ is a \textbf{germ of wave fronts} if $\tilde f$ is an immersion.
	We also call the image of a germ of wave fronts a \textbf{wave front}.
\end{definition}

Other authors such as \cite{Arnold_I} refer to $f$ as a \emph{Legendrian map germ}, and its image as a \emph{front}.

Let $\ms{O}_n$ be the ring of germs of functions on $(\mb{C}^n,0)$.
Given a smooth map germ $f\colon (\mb{C}^n,S) \to (\mb{C}^{n+1},0)$, we define the \textbf{ramification ideal} of $f$ as the ideal $\mc{R}(f) \subseteq \ms{O}_n$ generated by the $n\times n$ minors of the Jacobian matrix of $f$,
\[\begin{pmatrix}
\dfrac{\p f_1}{\p x_1}	&	\dots		& \dfrac{\p f_{n+1}}{\p x_1}	\\
\vdots			&	\ddots	& \vdots 					\\
\dfrac{\p f_1}{\p x_n}	&	\dots		& \dfrac{\p f_{n+1}}{\p x_n}
\end{pmatrix}\]

\begin{proposition}[\cite{Ishikawa_Survey}]\label{jacobian criterion frontal}
	Let $f\colon (\mb{C}^n,S) \to (\mb{C}^{n+1},0)$ be a smooth map germ with ramification ideal $\mc{R}(f)$.
	Then $f$ is a frontal map germ if and only if $\mc{R}(f)$ is a principal ideal.
\end{proposition}

\begin{remark}
	Let $f$ be a proper frontal map germ, $m_1,\dots,m_{n+1} \in \ms{O}_n$ the minors of the Jacobian matrix of $f$ and assume that $\mc{R}(f)=\langle m_j\rangle$.
	It follows from the proof of Proposition \ref{jacobian criterion frontal} that the Nash lift of $f$ is Legendrian equivalent to the map germ $\tilde f\colon (\mb{C}^n,S) \to \mb{C}^{2n+1}$ given by
		\[\tilde f(x)=\left(f(x),\frac{m_1(x)}{m_j(x)},\dots,\widehat{\frac{m_j(x)}{m_j(x)}},\dots,\frac{m_{n+1}(x)}{m_j(x)}\right),\]
	where $\widehat{\cdot}$ denotes an omitted component.
	This means that $\tilde f$ can be taken as the Nash lift of $f$.
\end{remark}

\begin{example}
	\begin{enumerate}
		\item Let $\gamma\colon (\mb{C},0) \to (\mb{C}^2,0)$ be the analytic plane curve given by $\gamma(x)=(p(x),q(x))$, and assume $\ord q \geq \ord p$.
		We have $\mc{R}(\gamma)=\langle p' \rangle$, so $\gamma$ is a frontal map germ.

		\item The folded Whitney umbrella can be parametrized as
		\begin{funcion*}
			f\colon (\mb{C}^2,0) \arrow[r] & (\mb{C}^3,0)\\
			(x,y) \arrow[r, maps to] & (x,y^2,xy^3)
		\end{funcion*}
		We have $\mc{R}(f)=\langle y\rangle$, so $f$ is a frontal map germ.
		
		\item The $F_4$ singularity from Mond's classification \cite{Mond_Classification} can be parametrised as
		\begin{funcion*}
			f\colon (\mb{C}^2,0) \arrow[r] & (\mb{C}^3,0)\\
			(x,y) \arrow[r, maps to] & (x,y^2,y^5+x^3y)
		\end{funcion*}
		We have $\mc{R}(f)=\langle x^3,y\rangle$, which is not a principal ideal, so $f$ is not a frontal.
	\end{enumerate}
\end{example}

We now state a series of definitions and results that we shall use throughout this paper.
Proofs for these statements can be found in \cite{MNO_Frontals}.

We say that two smooth map germs $f,g\colon (\mb{C}^n,S) \to (\mb{C}^p,0)$ (not necessarily frontal) are $\ms{A}$-equivalent if there are diffeomorphisms $\phi\colon (\mb{C}^n,S) \to (\mb{C}^n,S)$ and $\psi\colon (\mb{C}^p,0) \to (\mb{C}^p,0)$ such that $g=\psi\circ f\circ \phi^{-1}$.
We also define the space
	\[T\ms{A}_ef=\left\{\left.\frac{df_t}{dt}\right|_{t=0}: f_t=\psi_t\circ f\circ \phi_t^{-1}\right\},\]
where $\phi_t\colon (\mb{C}^n,S) \to (\mb{C}^n,S)$ and $\psi_t\colon (\mb{C}^p,0) \to (\mb{C}^p,0)$ are families of diffeomorphisms and $t \in \mb{D}$.
This space can be computed algebraically as the sum $tf(\theta_n)+\omega f(\theta_p)$, where
\begin{funcion} \label{t y omega}
	tf\colon \theta_n \arrow[r] 	& \theta(f) 		&& \omega f\colon \theta_p \arrow[r] 	& \theta(f) \\
	\xi \arrow[r, maps to] 		& df\circ \xi 	&& 	\eta \arrow[r, maps to]		& \eta \circ f
\end{funcion}

\begin{proposition}\label{frontals preserved under A}
	Let $f,g\colon (\mb{C}^n,S) \to (\mb{C}^{n+1},0)$ be $\ms{A}$-equivalent holomorphic map germs.
	If $f$ is frontal, $g$ is frontal.
	Moreover, if $f$ is a germ of wave front, $g$ is a germ of wave front.
\end{proposition}

Given a frontal map germ $f\colon (\mb{C}^n,S) \to (\mb{C}^{n+1},0)$, we define the space or infinitesimal frontal deformations of $f$ as
	\[\ms{F}(f):=\left\{\left.\frac{df_t}{dt}\right|_{t=0}: f_0=f, (f_t,t) \text{ frontal}\right\}\]
Using Proposition \ref{frontals preserved under A}, we see that $T\ms{A}_ef \subseteq \ms{F}(f)$.
In particular, this implies that $tf(\xi)$ and $\omega f(\eta)$ are in $\ms{F}(f)$ for all $\xi \in \theta_n$ and $\eta \in \theta_{n+1}$.

\begin{definition}
	We define the \textbf{frontal codimension} or $\ms{F}$-codimension of $f$ as
		\[\codim_{\ms{F}e}(f):=\dim_\mb{C}\frac{\ms{F}(f)}{T\ms{A}_ef}\]
	We say $f$ is $\ms{F}$-finite if $\codim_{\ms{F}_e}(f) < \infty$.
\end{definition}

Given a proper frontal map-germ $f\colon (\mb{C}^n,S) \to (\mb{C}^{n+1},0)$, we define the corank (resp. integral corank) of $f$ as the maximum corank (resp. integral corank) among its branches.

\begin{definition}\label{frontal unfolding}
	A \textbf{frontal unfolding} of a frontal map germ $f\colon (\mb{C}^n,S) \to (\mb{C}^{n+1},0)$ is an unfolding $F\colon (\mb{C}^d\times \mb{C}^n,\{0\}\times S) \to (\mb{C}^d\times \mb{C}^{n+1},0)$ which is frontal as a map germ.
	A frontal map germ $f\colon (\mb{C}^n,S) \to (\mb{C}^{n+1},0)$ is stable as a frontal or $\ms{F}$-stable if every $d$-parameter frontal unfolding $F$ of $f$ is $\ms{A}$-equivalent to $f\times\id_{(\mb{C}^d,0)}$.
\end{definition}

\begin{remark}
	Let $f\colon (\mb{C}^n,S) \to (\mb{C}^{n+1},0)$ be a proper wave front germ (i.e. a proper frontal germ with integral corank $0$) with Nash lift $\tilde f\colon (\mb{C}^n,S) \to PT^*\mb{C}^{n+1}$, and $F\colon (\mb{C}^d\times \mb{C}^n,\{0\}\times S) \to (\mb{C}^d\times \mb{C}^{n+1},0)$ be a frontal unfolding in the form
		\[F(u,x)=(u,f_u(x)).\]
	The unfolding $F$ induces an integral $d$-parameter deformation $\tilde f_u$ of $\tilde f$.
	Since $f$ is a wave front, $\tilde f$ is an immersion, so $\tilde f_u$ (and thus the Nash lift $\tilde F$ of $F$) is an immersion.
	It follows that $F$ is a wave front.
\end{remark}

\section{Counting corank $2$ singularities}
Let $f\colon (\mb{C}^n,S) \to (\mb{C}^{n+1},0)$ be a generic germ of frontal map germs, and $f_t$ be a $1$-parameter deformation defining a stable frontal unfolding of $f$.
For $0 < |t| \ll 1$, the image of $f_t$ exhibits at most stable frontal singularities in a neighbourhood of the origin, the number of which constitutes an smooth invariant of $f$ and can thus be used to explore analytic invariants of $f$, such as the frontal Milnor number of $f$ (see \cite{MNO_Surfaces, MNO_Mond} for more information on this topic).

\begin{example}\label{swallowtail}
	In this example, we show how one can compute an analytic invariant of a wave front singularity $f\colon (\mb{C}^2,0) \to (\mb{C}^3,0)$ (the frontal Milnor number of $f$, as defined on \cite{MNO_Mond}) using the number of isolated singularities in a generic frontal perturbation of $f$.
	Consider the wave front $f\colon (\mb{C}^2,0) \to (\mb{C}^3,0)$ given by
	\begin{equation} \label{n swallowtails}
		f(x,y)=(x,2y^3+x^ny,3y^4+x^ny^2),
	\end{equation}
	and the $1$-parameter family of wave fronts
		\[f_t(x,y)=(x,2y^3+x^ny+ty,3y^4+x^ny^2+ty^2).\]
	The mapping $F(t,x,y)=(t,f_t(x,y))$ is a stable frontal unfolding of $f$, and the image of $f_t$ contains $n$ swallowtail singularities (see Figure \ref{swallowtails} for the case $n=2$).
	
	\begin{figure}[ht]
	\includegraphics[width=.4\textwidth]{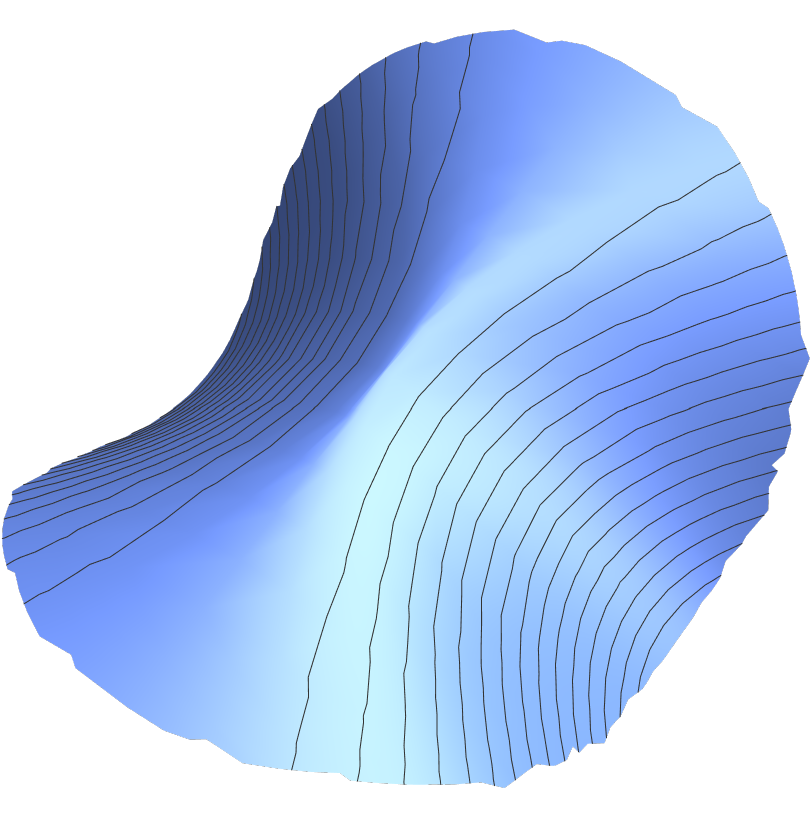}
	\includegraphics[width=.4\textwidth]{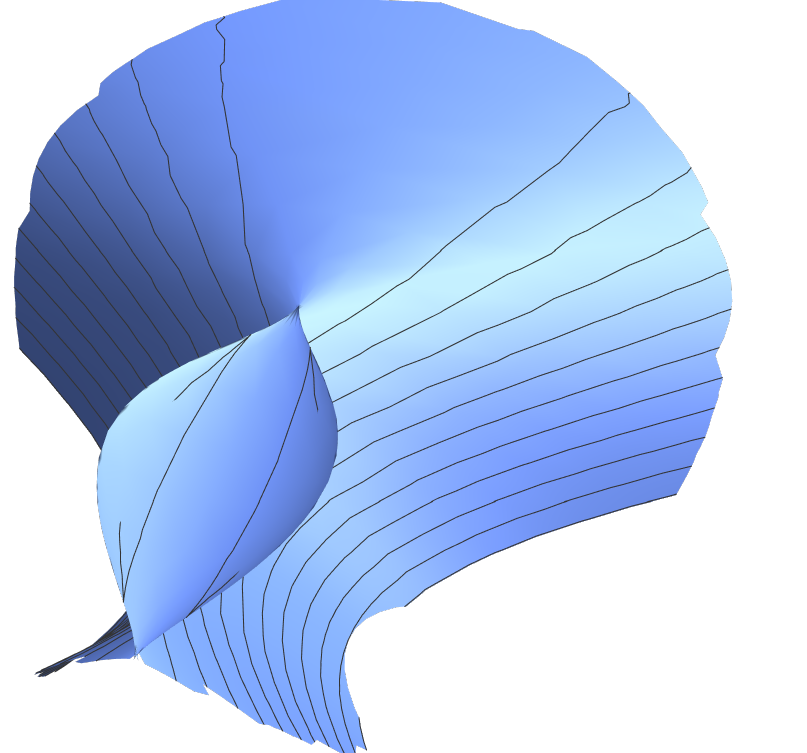}
	\caption{Real image of the mapping \eqref{n swallowtails} for $n=2$ (left) along with the real image of $f_t$ for $t=1$ (right), where a sphere has emerged.
	This wave front can be obtained as the discriminant of the family $4^2_1$ from Marar and Tari's classification \cite{MararTari}.
 \label{swallowtails}}
\end{figure}
	
	Using the results from \cite{MNO_Surfaces}, the frontal Milnor number of $f$ is given by\footnote{The expression for the number of spheres presented in \cite{MNO_Surfaces} had a mistake in the final result; the one given here is the correct expression.}
		\[\mu_\ms{F}(f)=\frac{1}{2}\Big(\mu(D_+(f),0)+S+W-1\Big)-K-2T,\]
	where $D_+(f)$ denotes the source curve of transverse double points, and $S$, $W$, $K$ and $T$ denote the number of swallowtails, folded Whitney umbrellas, cuspidal double points and transverse triple points in a generic frontal perturbation of $f$.
	The image equation for $D_+(f)$ is given in this case by $2y^2+x^n=0$, hence $\mu(D_+(f),0)=n-1$.
	Moreover, we have that $S=n$ and $W=K=T=0$, as $f$ is an augmentation of a swallowtail singularity.
	Therefore, the frontal Milnor number of $f$ is given by
		\[\frac{1}{2}\Big(\mu(D_+(f),0)+S-1\Big)=\frac{1}{2}(n-1+n-1)=n-1.\]
	This means that for $0 < |t| \ll 1$, the image of $f_t$ is homotopically equivalent to a bouquet of $n-1$ spheres.
\end{example}

If $f\colon (\mb{C}^n,S) \to (\mb{C}^{n+1},0)$ is a corank $1$ wave front, the image of $f$ can be seen as the discriminant of a smooth $h\colon (\mb{C}^{n+1},S) \to (\mb{C}^{n+1},0)$, and we can apply the formulas given by Marar, Montaldi and Ruas in \cite{MMR_Schemes} to count their stable isolated singularities.
A method to compute one such $h$ is given in \cite{MNO_Mond}.
On the other hand, if $f\colon (\mb{C}^2,S) \to (\mb{C}^3,0)$ is a frontal map germ with corank $1$, we give in \cite{MNO_Surfaces} expressions to count their stable isolated singularities.
In the following sections, we explore formulas to count stable singularities on corank $2$ wave fronts.

The classification of stable wave front singularities $f\colon (\mb{C}^n,S) \to (\mb{C}^{n+1},0)$ \cite{Arnold_I} establishes that corank $2$ stable singularities only emerge for $n \geq 3$, the first of which is the $D_4$ singularity for $n=3$.
The $D_4$ singularity can be parametrised as the map germ $f\colon (\mb{C}^3,0) \to (\mb{C}^4,0)$ given by
	\[f(u,v,w)=(u, v w, 2 u v + 3 v^2 + w^2, u v^2 + 2 v^3 + 2 v w^2).\]
The other stable families for $n=3$ are again the $A_k$ singularities with $k \leq n+1$, all of which have corank $1$.
Therefore, counting corank $2$ singularities in the stable deformation of a wave front equates to counting $D_4$ singularities.

Given a corank $2$ wave front $f\colon (\mb{C}^3,0) \to (\mb{C}^4,0)$, we denote its Jacobian matrix by $J_f$.
The points where $D_4$ singularities emerge in a $d$-parameter unfolding $F$ of $f$ are given by the points where $F$ has corank $2$ for a fixed value of $t$; this is, the points where the $(d+2)\times (d+2)$ minors of $J_F$ vanish.
We shall denote the ideal generated by the $k\times k$ minors of a matrix $M$ as $I_k(M)$.

Let $d \geq 0$. We can choose coordinates in the source and target such that a corank $2$ map germ $F\colon (\mb{C}^d\times\mb{C}^2,0) \to (\mb{C}^d\times \mb{C}^3,0)$ takes the form
	\begin{equation}
		F(u,v,w)=(u,p(u,v,w),q(u,v,w),r(u,v,w)), \label{corank 2 explicit}
	\end{equation}
with $u \in \mb{C}^d$, $v,w \in \mb{C}$ and $p,q,r \in \mf{m}_{3+d}^2$.
In this case, the Jacobian matrix of $f$ adopts the form
	\[
	J_F=
	\begin{pmatrix}
		1		& 0		& \dots 	& 0		& 0		& 0		\\
		\vdots	& \vdots 	& \ddots	& \vdots	& \vdots	& \vdots	\\
		0		& 0		& \dots	& 1		& 0		& 0		\\
		p_{u_1}	& p_{u_2}	& \dots	& p_{u_d}	& p_v	& p_w	\\
		q_{u_1}	& q_{u_2}	& \dots	& q_{u_d}	& q_v	& q_w	\\
		r_{u_1}	& r_{u_2}	& \dots	& r_{u_d}	& r_v		& r_w
	\end{pmatrix}
	\]
Computing the minors by blocks, it is then easy to see that
	\[I_{d+1}(J_F)=\langle p_v, p_w, q_v, q_w, r_v, r_w\rangle.\]
Nevertheless, the terms $r_v$ and $r_w$ can be omitted by using that $F$ is a wave front:

\begin{lemma}\label{lemma alpha beta}
	There are $\alpha, \beta \in \ms{O}_n$ such that
	\begin{align}
		r_v=\alpha p_v+\beta q_v; && r_w=\alpha p_w+\beta q_w \label{expressions rv rw}
	\end{align}
	and the matrix
		\[M(u,v,w)=
		\begin{pmatrix}
			\alpha_v(u,v,w) 	& \alpha_w(u,v,w)	\\
			\beta_v(u,v,w)	& \beta_w(u,v,w)
		\end{pmatrix}\]
	is invertible for all $(u,v,w)$ in an open neighbourhood $U \subseteq \mb{C}^n$ of $0$.
\end{lemma}

\begin{proof}
	Since $F$ is a frontal map germ, we can choose coordinates in the source and target such that $F$ verifies the differential equation
	\begin{equation} \label{differential equation corank 2}
		dr=\lambda_1\,du_1+\dots+\lambda_d\,du_d+\alpha\,dp+\beta\,dq.
	\end{equation}
	for some $\lambda_1,\dots,\lambda_n,\alpha,\beta \in \ms{O}_n$.
	Taking the coefficients with respect to $dv$ and $dw$, we see that
	\begin{align}
		r_v=\alpha p_v+\beta q_v;		&&	r_w=\alpha p_w+\beta q_w. \label{expressions rv rw}
	\end{align}
	It then follows from \cite{MNO_Mond}, Lemma 2.17 that the matrix $M(u,v,w)$ is invertible in a neighbourhood $U \subseteq \mb{C}^n$ of $0$, as claimed.
\end{proof}

We will now give an expression to count the number of corank $2$ points in the image of $F$ that depends only on $f$, using a result from analytic geometry known as the \emph{conservation of multiplicity} principle.
The conservation of multiplicity principle states the following:

\begin{theorem}[\cite{MondNuno}, Corollary E.6]\label{conservation}
	Let $x \in \mb{C}^n$, $M$ be a Cohen-Macaulay $\ms{O}_{n,x}$-module and $\ms{M}$ a representative of $M$ (in the sense of \cite{MondNuno}, \S E.4).
	Then, there is an open neighbourhood $U \subseteq \mb{C}^n$ of $x$ such that
		\[\dim_\mb{C} M=\sum_{x \in \supp(\ms{M})\cap U} \dim_\mb{C} \ms{M}_x\]
\end{theorem}
 
If the module $\ms{O}_3/I_2(J_f)$ is Cohen-Macaulay, then we can apply Theorem \ref{conservation} to count the number of corank $2$ points in a stable frontal unfolding $F$ using only $\ms{O}_3/I_2(J_f)$.
We believe this is the case for any $\ms{F}$-finite frontal map germ $f$ (wave front or otherwise); nonetheless, we have not been able to prove it, so we will instead define a complete intersection module which counts the number of corank $2$ points of $F$ when $f$ is a wave front.

\begin{example}
	Let $f\colon (\mb{C}^3,0) \to (\mb{C}^4,0)$ be the smooth map germ given by
		\[f(u,v,w)=(u, v w, 2 u^n v + 3 v^2 + w^2, u^n v^2 + 2 v^3 + 2 v w^2).\]
	We have $\ms{R}(f)=\langle u^nv+3v^2-w^2\rangle$ and $\ic(f)=0$, hence $f$ is a wave front.
	We also have that $I_2(J_f)=(u^n,v,w)$, so $\ms{O}_3/I_2(J_f)$ is a complete intersection module with dimension $n$.
	Therefore, a stable frontal unfolding of $F$ has $n$ corank $2$ points.
	In particular, the germ of $F$ at those points describes an $D_4$ singularities.
\end{example}

Let $F\colon (\mb{C}^n,0) \to (\mb{C}^{n+1},0)$ be a germ of wave front with corank $2$: the image of $F$ is the discriminant of a smooth, corank $1$ map germ $H\colon (\mb{C}^{n+2}, 0) \to (\mb{C}^{n+1},0)$ with smooth critical set \cite{Arnold_I, MNO_Mond}.
We may choose coordinates in the source and target such that
	\begin{align}	\label{H parametrised}
		H(u,s,t,v,w)=(u,s,t,h(u,s,t,v,w)); && u \in \mb{C}^{n-2},\; s,t,v,w \in \mb{C}.
	\end{align}
In particular, we note that $H$ is an unfolding of a certain function $(v,w) \mapsto h(0,0,0,v,w)$.
This means that $F$ has corank $2$ at $(u,v,w) \in \mb{C}^n$ if and only if there are $s,t \in \mb{C}$ such that $x=(s,t,u,v,w)$ is a critical point of $h$ and the Hessian matrix of $h$ has corank $2$ at $x$:
\begin{equation}\label{hessian h}
	h_v(x)=h_w(x)=h_{vv}(x)=h_{vw}(x)=h_{ww}(x)=0.
\end{equation}
Note that if $F$ is a stable unfolding of a $f\colon (\mb{C}^3,0) \to (\mb{C}^4,0)$, $h$ is in fact the $D_4$ singularity from Arnold's classification of simple functions \cite{Arnold_72}.
Alternatively, we note that if $H$ has corank $2$ at a point $x$, then $x$ is a $\Sigma^{2,2}$ point of $H$ in Thom-Boardman notation (see \cite{MondNuno}, \S 4.7 for details).
Fukui, Nuño-Ballesteros and Saia \cite{FNBS_98} showed that the locus of $\Sigma^{2,2}$ points of $H$ is given by the zeros of a certain ideal they call $J_{2,2}(H)$.
In this case,
	\[J_{2,2}(H)=\langle h_v,h_w, h_{vv}, h_{vw}, h_{ww}\rangle,\]
whose zero locus is the set of solutions of Equation \eqref{hessian h}.

Using a result from \cite{MNO_Mond}, we can write $h$ explicitly as
	\[h(u,s,t,v,w)=r(u,v,w)+\alpha(u,v,w)(s-p(u,v,w))+\beta(u,v,w)(t-q(u,v,w)),\]
where $s,t \in \mb{C}$.
Then, Equation \eqref{hessian h} becomes
\begin{align}
	0=h_v	&=\alpha_v(s-p)+\beta_v(t-q);	\label{d4:hv}	\\
	0=h_w	&=\alpha_w(s-p)+\beta_w(t-q);	\label{d4:hw}	\\
	0=h_{vv}	&=\alpha_{vv}(s-p)-\alpha_vp_v+\beta_{vv}(t-q)-\beta_vq_v;	\label{d4:hvv}\\
	0=h_{vw}	&=\alpha_{vw}(s-p)-\alpha_vp_w+\beta_{vw}(t-q)-\beta_vq_w;\\
	0=h_{ww}	&=\alpha_{ww}(s-p)-\alpha_wp_w+\beta_{ww}(t-q)-\beta_wq_w. \label{d4:hww}
\end{align}
Equations (\ref{d4:hv}, \ref{d4:hw}) form a linear system of equations, with the coefficient matrix being the matrix $M$ from Lemma \ref{lemma alpha beta}, meaning that $s-p=t-q=0$.
Applying this information to (\ref{d4:hvv}\textendash\ref{d4:hww}), we arrive at
\begin{align*}
	0=\beta_vq_v+\alpha_vp_v;	&&
	0=\beta_vq_w+\alpha_vp_w;	&&
	0=\beta_wq_w+\alpha_wp_w,
\end{align*}
which yields the following

\begin{theorem}\label{count D4}
	Let $F\colon (\mb{C}^n,0) \to (\mb{C}^{n+1},0)$ be a germ of wave front with corank $2$.
	Take coordinates in the source and target such that $F$ is given as in \eqref{corank 2 explicit} and verifies the identity \eqref{differential equation corank 2}.
	Then, the set of points $(u,v,w) \in \mb{C}^n$ where $F$ has corank $2$ near $S$ is given the common zeros of the functions
		\[\beta_vp_v+\alpha_vq_v, \beta_vp_w+\alpha_vq_w, \beta_wp_w+\alpha_wq_w.\]
	Therefore, $I_{d+1}(J_F)$ has three generators and $\ms{O}_3/I_2(J_F)$ is a complete intersection module for $n=3$.
\end{theorem}

Now let $f\colon (\mb{C}^3,0) \to (\mb{C}^4,0)$ be an $\ms{F}$-finite germ of wave front with stable frontal unfolding $F\colon (\mb{C}^d\times\mb{C}^3,0) \to (\mb{C}^d\times \mb{C}^4,0)$.
Since $f$ is $\ms{F}$-finite, $F$ is also $\ms{F}$-finite, so we can find a representative $F\colon U \to V$ of $F$ such that the restriction $F|\colon U \backslash S \to V\backslash \{0\}$ only exhibits frontal stable singularities (\cite{MNO_Frontals}, Theorem 4.6). 
In particular, the points in $U\backslash S$ where $F$ has corank $2$ will be $D_4$ singularities, as it is the only stable singularity of corank $2$ in dimension $3$.
Combining this with Theorem \ref{conservation}, we arrive to the following

\begin{corollary}
	The number of $D_4$ singularities near the origin on the image of a stable frontal unfolding of $f$ is equal to
		\[\dim_\mb{C}\frac{\ms{O}_3}{\langle p_v, p_w, q_v, q_w\rangle},\]
	where $g_v$, $g_w$ denote the partial derivatives of the function $g$ with respect to $v$ and $w$.
\end{corollary}

\begin{example}
	Let $f\colon (\mb{C}^3,0) \to (\mb{C}^4,0)$ be the map germ given by
		\[f(u,v,w)=(u, v w ,2 u^n v + 3 u v^2 + 4 v^3 + w^2, u^n v^2 + 2 u v^3 + 3 v^4 + 2 v w^2).\]
	This mapping can be obtained as a pullback of the mapping $(v,w) \mapsto (v w, 4 v^3 + w^2, 3 v^4 + 2 v w^2)$, which unfolds into a $D_5$ singularity.
	Therefore, it is not only a wave front, but also $\ms{F}$-finite.
	
	The ideal $I_2(J_f)$ is equal to $\langle u,v^2,w\rangle$, hence the number of $D_4$ singularities in a stable frontal deformation of $f$ is equal to $2$, regardless of $n$.
\end{example}




\end{document}